\documentclass[10pt]{amsart}
\usepackage{amscd}
\usepackage{amssymb}
\newtheorem{Theorem}{Theorem}[section]
\newtheorem{Lemma}[Theorem]{Lemma}
\newtheorem{Proposition}[Theorem]{Proposition}
\newtheorem{Corollary}[Theorem]{Corollary}

\def\ot{\otimes}
\def\ep{\varepsilon}

\def\gr{\operatorname{gr}}
\def\Hom{\operatorname{Hom}}
\def\Reg{\operatorname{Reg}}

\def\Aut{\operatorname{Aut}}

\def\Alg{\operatorname{Alg}}

\def\co-Lift{\operatorname{co-Lift}}
\def\and{\operatorname{and}}

\def\im{\operatorname{im}}

\def\m{\operatorname{m}}

\def\H{H}

\def\del{\partial}

\def\Der{\operatorname{Der}}
\def\im{\operatorname{im}}
\def\H{\operatorname{\mathcal H}}
\def\exp{\operatorname{exp}}
\def\Exp{\operatorname{Exp}}
\def\im{\operatorname{im}}
\newtheorem{Conjecture}{Conjecture}

\def\setst#1#2{\left\{\left. #1 \right| #2 \right\}}

\title{Cohomological aspects of Hopf algebra liftings}

\author{L. Grunenfelder}

\address{Department of Mathematics, The University of British Columbia, 
Vancouver, BC V6T 1Z2, Canada}

\email{luzius@math.ubc.ca, luzius@mathstat.dal.ca}

\begin{document}

\tolerance=1000 \overfullrule 5pt

\begin{abstract} A recent result of ours \cite{GM} shows that all Hopf algebra liftings of a given diagram in the sense of Andruskiewitsch and Schneider are cocycle deformations of each other. Here we develop a \lq non-abelian' cohomology theory, which gives a method for an explicit description of cocycles relevant to the lifting process.  \end{abstract}

\maketitle

\setcounter{section}{-1}

\section{Introduction}

The Nichols algebra $B(V)$ of a crossed $kG$-module $V$ is a connected braided Hopf algebra. In terms of generators and relations it can be described via a certain pushout diagram
$$\begin{CD}
K(V) @>\kappa >> R(V) \\
@V\ep VV  @V\pi VV \\
k @>\iota >>  B(V)
\end{CD}$$
of connected braided Hopf algebras. The Radford biproduct or bosonization $H(V)=B(V)\# kG$ has a similar presentation in the category of ordinary Hopf algebras. A lifting of $H(V)$ is a pointed Hopf algebra $H$ for which $\gr^cH\cong H(V)$, where $\gr^cH$ is the graded Hopd algebra associated with the coradical filtration of $H$. Such liftings are obtained by deforming the multiplication of $H(V)$. In this context the lifting problem for $V$ is asking for the characterization and classification of all liftings of $H(V)$. This problem has been solved by Andruskiewitsch and Schneider in \cite{AS} for a large class of crossed $kG$-modules of finite Cartan type, which will carry the attribute \lq special' in this paper. This allows, in particular, for a classification of all finite dimensional pointed Hopf algebras $A$ for which the order of the abelian group of points is not divisible by any prime $< 11$. In recent work \cite{GM} we have shown that for any given $V$ in this class all liftings of $H(V)$ are cocycle deformations of each other (see also \cite{Ma1}, Appendix).  This is done via a description of the lifted Hopf algebras suitable for the application of  results by Masuoka about Morita-Takeuchi equivalence \cite{Ma} and by Schauenburg about Hopf-Galois extensions \cite{Sch}. For some special cases such results had been obtained in \cite{Ma, Di, BDR, Gr}. In addition, our results in \cite{GM} show that every lifting of $H(V)$, and therefore the corresponding cocycle, is completely determined by a $G$-invariant algebra map $f\in\Alg_G(K(V),k)$, but without  an explicit description of the corresponding cocycle in terms of $f$. 

In the present paper we aim at making this connection between the $G$-invariant algebra map $f\in\Alg_G(K,k)$ and the corresponding deforming cocycle $\sigma :B\ot B\to k$ more explicit. For that purpose we first describe a non-abelian equivariant cohomology theory for braided Hopf algebras $X$  in the category of crossed $H$-modules and for their bosonizations $X\# H$, where $H$ is an ordinary Hopf algebra. The Radford biproduct $X\# H$ is an ordinary Hopf algebra and carries the obvious $H$-bimodule structure. A pushout diagram  of (braided) Hopf algebras as above, in which $\kappa $ has a $H$-module coalgebra retraction gives rise to a Meier-Vietories type 5-term exact sequence 
\begin{eqnarray*}
1 \to \Alg_H(B,k) {\buildrel\pi^*\over\longrightarrow} \Alg_H(R,k) {\buildrel\kappa^*\over\longrightarrow} \Alg_H(K,k) {\buildrel\delta\over\longrightarrow} \H_H^2(B,k) {\buildrel\pi^*\over\longrightarrow} \H_H^2(R,k)
\end{eqnarray*}
of pointed sets. In the situation of the lifting problem,  when $B$ is the Nichols algebra of a crossed $kG$-module of special finite Cartan type, then $\Alg_G(R,k)$ is trivial. If, in addition, $K$ is a $K$-bimodule coalgebra retract in $R$, then the connecting map $\delta :\Alg_G(K(V),k)\to \H^2_G(B(V),k)$ exists and is injective. Then, in view of the characterization of  liftings in \cite{AS, GM}, the cocycles obtained via the connecting map account for all liftings of $B(V)\# kG$.

The  5-term sequence for equivariant Hochschild cohomology
\begin{eqnarray*}
0 \to \Der_H(B,k) {\buildrel\pi^*\over\longrightarrow} \Der_H(R,k) {\buildrel\kappa^*\over\longrightarrow} \Der_H(K,k) {\buildrel\delta\over\longrightarrow} H_H^2(B,k) {\buildrel\pi^*\over\longrightarrow} H_H^2(R,k)
\end{eqnarray*}
has been established in \cite{GM} and is an exact sequence of  vector spaces. Here it suffices that $K$ is a $K$-bimodule retract in $R$, which in the liftings situation is always the case. The question about the relationship between Hochschild cohomology and non-abelian cohomology naturally arises in this context. In the cocommutative case there are Sweedler's results. For quantum linear spaces, i.e: for diagrams of type $A_1\times A_1\times\ldots\times A_1$, there is an exponential relationship between Hochschild cocycles and  those \lq multiplicative' cocycles which depend on the root vector parameters alone \cite{GM}. Here we present some more general results on this topic involving linking as well. This includes an approach to quantum planes quite different from that of  \cite{ABM}, Section 5. In the last section we also develop a program for the connected case, and apply it to diagrams of type $A_2$. Results for type $A_n$, $n>2$, and for type $B_2$ will be part of a forthcoming paper.

The notation in the paper as in \cite{GM} is pretty much standard; $m:A\ot A\to A$ denotes multiplication, $\Delta :C\ot C\to C$ comultiplication, $s:H\to H$ the antipode, and $*:\Hom (C,A)\ot\Hom (C,A)\to \Hom (C,A)$ the convolution multiplication $f*f'=m(f\ot f')\Delta$. We use Sweedler's notation in the form $\Delta (c)= c_1\ot c_2$ etc., and also $\Delta^{(n)}=(1\ot\Delta^{(n-1)})\Delta$ for $n\ge 1$ with $\Delta^{(0)}=1$. The notaion used for coactions  of a Hopf algebra $\delta :X\to H\ot X$ is $\delta (x)=x_{-1}\ot x_0$.

\section{A non-abelian cohomology}

Every lifting of the bosonisation $A=B\# kG$ of the Nichols algebra $B$ of a finite dimensional special crossed $G$-module $V$ is determined by a $G$-invariant algebra map $f\in_G\Alg_G(K\# kG,k)$, and it is also a cocycle deformation $A_{\sigma}$ of $A$. The $G$-invariant \lq multiplicative' cocycle $\sigma :A\ot A\to k$ must therefore be completely determined by the $G$-invariant algebra map $f:K\to k$. In the examples presented in \cite{GM} Section 3 the relation between the two entities is given explicitly. In this paper non-abelian cohomology will serve to clarify this relationship for some special diagrams of finite Cartan type. 

\subsection{The \lq multiplicative' cohomology} The non-abelian equivariant cohomology of a braided Hopf algebra in the category of crossed $H$-modules $X$ or its bosonization, which is an ordinary Hopf algebra, is defined via the cosimplicial group complex of regular elements
$$\begin{array}{ccccccc}
\Reg_H(k,k) & ^{\del^0\atop{\longrightarrow}}_{\del^1\atop{\longrightarrow}} & \Reg_H (X,k) & ^{{\del^0\atop{\longrightarrow}}\atop{\del^1\atop{\longrightarrow}}}_{\del^2\atop{\longrightarrow}} & \Reg_H (X^2,k) & ^{{\del^0\atop{\longrightarrow}}\atop{\del^1\atop{\longrightarrow}}}_{{\del^2\atop{\longrightarrow}}\atop{\del^3\atop{\longrightarrow}}} & \Reg_H\ (X^3,k)
\end{array}$$
in the standard cosimplicial algebra complex
$$\begin{array}{ccccccc}
\Hom_H(k,k) & ^{\del^0\atop{\longrightarrow}}_{\del^1\atop{\longrightarrow}} & \Hom_H (X,k) & ^{{\del^0\atop{\longrightarrow}}\atop{\del^1\atop{\longrightarrow}}}_{\del^2\atop{\longrightarrow}} & \Hom_H (X^2,k) & ^{{\del^0\atop{\longrightarrow}}\atop{\del^1\atop{\longrightarrow}}}_{{\del^2\atop{\longrightarrow}}\atop{\del^3\atop{\longrightarrow}}} & \Hom_H\ (X^3,k) ,
\end{array}$$
where $X^i$ denotes the i-th tensor power of $X$, and where
$$\del^if = \left\{ \begin{array}{ll}
\ep\ot f & \mbox{if $i=0$} \\
f(1^{i-1}\ot m\ot 1^{n-i-1}) & \mbox{if $0<i<n$} \\
f\ot\ep  & \mbox{if $i=n$}
\end{array} \right. $$
are the standard cofaces.

The first equivariant \lq non-abelian' cohomology of $X$ with coefficients in $k$ is given by
$$\H^1_H(X,k)=Z^1_H(X,k)=\{ f\in\Reg_H(X,k)|\del^1f=\del^2f*\del^0f\} =\Alg_H(X,k)$$
which is a group under the convolution multiplication. A 1-cocycle is therefore an element $f\in\Reg_H(X,k)$ such that $fm=(f\ot\ep )*(\ep\ot f)=m_k(f\ot f)$, that is an algebra map. For the second cohomology define the set of \lq non-abelian' 2-cocycles by
$$Z^2_H(X,k)=\{ \sigma\in\Reg_H(X^2,k)|\del^0\sigma *\del^2\sigma =\del^3\sigma *\del^1\sigma , \sigma (\iota\ot 1)=\ep =\sigma (1\ot\iota ) \}$$
which means that $\sigma\in\Reg_H(X^2,k)$ is a cocycle if and only if the \lq multiplicative' 2-cocycle conditions
$$(\ep\ot\sigma ) *\sigma (1\ot m) =(\sigma\ot\ep )*\sigma (m\ot 1) \ ,\ \sigma (\iota\ot 1)=\ep =\sigma (1\ot\iota )$$
are satisfied, in particular $\sigma (y_1\ot z_1)\sigma (x\ot y_2z_2)=\sigma (x_1\ot y_1)\sigma (x_2y_2\ot z)$ in the ordinary case and $\sigma (y_1\ot (y_2)_{-1}z_1)\sigma (x\ot (y_2)_0z_2)=\sigma (x_1\ot (x_2)_{-1}y_1)\sigma ((x_2)_0y_2\ot z)$ in the braided case.
Define a relation on $\Reg_H(X^2,k)$ by declaring $\sigma\sim\sigma '$ if and only if $\sigma '=\del^0\chi *\del^2\chi *\sigma *\del^1\chi^{-1}$ for some $\chi\in\Reg_H(X,k)$.

\begin{Lemma} \label{nonab} The relation $\sim $ just defined on $\Reg_H(X^2,k)$ is an equivalence relation, which restricts to $Z^2_H(X,k)$. The second "non-abelian" cohomology $\H^2_H(X,k)=Z^2_H(X,k)/\sim$ is a pointed set with distinguished element $class (\ep\ot\ep )=\im (\del :\Reg_H(X,k)\to  \Reg_H(X\ot X,k))$, where $\del f=\del^0f *\del^2f *\del^1f^{-1}$. Moreover, there is a natural isomorphism $\H^1_H(X,k)\cong \H^1_H(X\# H,k)$ and a natural injection $\H^2_H(X,k)\to \H^2_H(X\# H,k)$ for braided Hopf algebras $X$ in the category of crossed $H$-modules and their bosonisations $Y=X\# H$.  
\end{Lemma}

\begin{proof} First we will show that it is sufficient to prove the assertions for ordinary Hopf algebras. If $X$ is a Hopf algebra in the category of crossed $H$-modules then $Y=X\# H$ is an ordinary Hopf algebra.
The linear map
$$\psi_n:Y^n\to X^n$$
defined inductively by $\psi_1(xh)=x\ep (h)$ and $\psi_n(xh\ot y)=x\ot h\psi_{n-1}y $ is a $H$-bimodule map (diagonal left and trivial right $H$-action on $X^n$), which has linear right inverse $\phi_n: X^n\to Y^n$ given by $\phi_1(x)=x1$ and $\phi_n(x\ot y)=x1\ot\phi_{n-1}y$. It factors through $Y^{(n)}=Y\ot_HY\ot_H\ldots\ot_HY$ to give a left $H$-module isomorphism $Y^{(n)}\ot_Hk\cong X^n$. Induction on $n$ shows that it is also compatible with the \lq coalgebra structures' in that $\Delta_{X^n}\psi_n=(\psi_n\ot\psi_n)\Delta_{Y^n}$ and $\ep\psi_n=\ep$. The induced injective algebra map
$$\psi^n:\Hom_H(X^n,k)\to\Hom_H(Y^n,k)$$
is then given by $\psi^n(f)=f\psi_n$, that is $\psi^n(f)(xh\ot y)=f(x\ot h\psi_{n-1}y)$ or $\psi^nf(x^1h^1\ot x^2h^2\ot\ldots\ot x^ng^n)=f(x^1\ot h^1_1x^2\ot\ldots\ot h^1_{n-1}h^2_{n-2}\ldots h^{n-1}x^n)$. It is an algebra map, since it preserves the convolution multiplication,
\begin{eqnarray*}
\psi^n(f*f')=(f*f')\psi_n=(f\ot f')\Delta_{X^n}\psi_n=(f\ot f')(\psi_n\ot\psi_n)\Delta_{Y^n} \\
=(f\psi_n\ot f'\psi_n)\Delta_{Y^n}=\psi^n(f)*\psi^n(f')
\end{eqnarray*}
and the convolution identity, $\psi^n(\ep )=\ep\psi_n=\ep$. It therefore automatically restricts to an injective group homomorphism
$$\psi^n:\Reg_H(X^n,k)\to\Reg_H(Y^n,k)$$
between the groups of regular elements. This leads to a injective homomorphism of the standard cosimplicial groups

$$\begin{array}{ccccccc}
\Reg_H(k,k) & ^{\del^0\atop{\longrightarrow}}_{\del^1\atop{\longrightarrow}} & \Reg_H (X,k) & ^{{\del^0\atop{\longrightarrow}}\atop{\del^1\atop{\longrightarrow}}}_{\del^2\atop{\longrightarrow}} & \Reg_H (X^2,k) & ^{{\del^0\atop{\longrightarrow}}\atop{\del^1\atop{\longrightarrow}}}_{{\del^2\atop{\longrightarrow}}\atop{\del^3\atop{\longrightarrow}}} & \Reg_H (X^3,k) \\
\|  & & \downarrow\psi^1 & & \downarrow\psi^2 & & \downarrow\psi^3 \\
\Reg_H(k,k) & ^{\del^0\atop{\longrightarrow}}_{\del^1\atop{\longrightarrow}} &\Reg_H (Y,k) & ^{{\del^0\atop{\longrightarrow}}\atop{\del^1\atop{\longrightarrow}}}_{\del^2\atop{\longrightarrow}} & \Reg_H (Y^2,k) & ^{{\del^0\atop{\longrightarrow}}\atop{\del^1\atop{\longrightarrow}}}_{{\del^2\atop{\longrightarrow}}\atop{\del^3\atop{\longrightarrow}}} & \Reg_H(Y^3,k)
\end{array}$$
compatible with the standard cofaces
$$\del^if = \left\{ \begin{array}{ll}
\ep\ot f & \mbox{if $i=0$} \\
f(1^{i-1}\ot m\ot 1^{n-i-1}) & \mbox{if $0<i<n$} \\
f\ot\ep  & \mbox{if $i=n$}
\end{array} \right. $$
in which $\psi^1$ is an isomorphism. It then suffices to prove the first assertion for the ordinary Hopf algebra $Y=X\# H$.

First observe that if $f\in\Reg_H(Y,k)$ then $\del f=\del^0f *\del^2f *\del^1f^{-1}$ is a $2$-cocycle:
\begin{eqnarray*}
\lefteqn{\bigl ((\ep\ot\del f)*\del f(1\ot m)\bigr )(x\ot y\ot z)}\\
&=&\del f(y_1\ot z_1)\del f(x\ot y_2z_2) \\
&=&f(z_1)f(y_1)f^{-1}(y_2z_2)f(y_3z_3)f(x_1)f^{-1}(x_2y_4z_4) \\
&=&f(x_1)f(y_1)f(z_1)f^{-1}(x_2y_2z_2) \\
&=&f(y_1)f(x_1)f^{-1}(x_2y_2)f(z_1)f(x_3y_3)f^{-1}(x_4y_4z_2) \\
&=&\del f(x_1\ot y_1)\del f(x_2y_2\ot z)\\
&=&(\del f\ot\ep )*\del f(m\ot 1)(x\ot y\ot z)
\end{eqnarray*}

Now we show that $\sim $ is an equivalence relation even on $\Reg_H(Y,k)$, and that it restricts to $Z_H^2(Y,k)$.

Reflexivity, $\sigma\sim\sigma$ of the relation $\sim $ obviously holds with $\chi =\ep$.

To check symmetry, observe that $\sigma '=\del^0\chi *\del^2\chi *\sigma *\del^1\chi^{-1}$ for some $\chi\in\Reg_H(Y,k)$ implies that $\sigma =\del^0\chi^{-1} *\del^2\chi^{-1} *\sigma '*\del^1\chi $ since $(\del^i\chi )^{-1}=\del^i\chi^{-1}$ and $\del^2\chi *\del^0\chi =\del^0\chi *\del^2\chi$. 

For transitivity suppose that in addition $\sigma ''=\del^0\psi *\del^2\psi *\sigma '*\del^1\psi^{-1}$ for some $\psi\in\Reg_H(Y,k)$. Then
\begin{eqnarray*}
\sigma ''&=&\del^0\psi *\del^2\psi *\del^0\chi *\del^2\chi *\sigma *\del^1\chi^{-1} *\del^1\psi^{-1} \\
&=&\del^0(\psi *\chi ) *\del^2(\psi *\chi ) *\sigma *\del^1(\psi *\chi )^{-1}
\end{eqnarray*}
since $\del^2\psi *\del^0\chi =\del^0\chi *\del^2\psi $ and the $\del^i$ are group homomorphisms.

To show that the equivalence relation $\sim$ restricts to $Z^2_H(Y,k)$ it suffices to show that if $\sigma\in Z^2_H(Y,k)$ and $\chi\in\Reg_H(Y,k)$ then $\sigma '=\del^0\chi *\del^2\chi *\sigma *\del^1\chi^{-1}$ is a cocycle as well:
\begin{eqnarray*}
\lefteqn{\bigl(\del^0\sigma '*\del^2\sigma '\bigr)(x\ot y\ot z)=\sigma '(y_1\ot z_1)\sigma '(x\ot y_2z_2)} \\
&=&\chi (z_1)\chi (y_1)\sigma (y_2\ot z_2)\chi^{-1}(y_3z_3)\chi (y_4z_4)\chi (x_1)\sigma (x_2\ot y_5z_5)\chi ^{-1}(x_3y_6z_6) \\
&=&\chi (x_1)\chi (y_1)\chi (z_1)\sigma (y_2\ot z_2)\sigma (x_2\ot y_3z_3)\chi^{-1}(x_3y_4z_4) \\
&=&\chi (x_1)\chi (y_1)\chi (z_1)\sigma (x_2\ot y_2)\sigma (x_3y_3\ot z_2)\chi^{-1}(x_4y_4z_3) \\
&=&\chi (y_1)\chi (x_1)\sigma (x_2\ot y_2)\chi^{-1}(x_3y_3)\chi (z_1)\chi (x_4y_4)\sigma (x_5y_5\ot z_2)\chi ^{-1}(x_6y_6z_3) \\
&=&\sigma '(x_1\ot y_1)\sigma '(x_2y_2\ot z)=(\del^3\sigma '*\del^1\sigma ')(x\ot y\ot z)
\end{eqnarray*}
and
$$\sigma '(x\ot 1)=\chi (x_1)\sigma (x_2\ot 1)\chi^{-1}(x_3)=\ep (x)=\chi (x_1)\sigma (1\ot x_2)\chi^{-1}(x_3)=\sigma '(1\ot x).$$
This proves the assertions for the bosonisation $Y=X\# H$. For the braided Hopf algebra $X$ they are now a consequence of the properties of the diagram above. It follows that
$$\Alg_H(X,k)=\H^1_H(X,k)\cong \H^1_H(Y,k)=\Alg_H(Y,k)$$
since $\psi^1$ is an isomorphism and $\psi^2$ is injective. Since, in addition, $\psi^3$ is injective as well it follows that $\sigma\in\Reg_H(X^2,k)$ is a 2-cocycle if and only if $\psi^2\sigma\in\Reg_H(Y^,k)$ is a cocycle. In particular, if $f\in\Reg_H(X,k)$ then $\del f=\del^0f *\del^2f *\del^1f^{-1}\in Z^2(X,k)$. Moreover, the following argument shows that the induced map $\H^2_H(X,k)\to \H^2(Y,K)$ is injective. Suppose that $\sigma , \sigma '\in Z^2(X,k)$ are such that $\psi^2\sigma \sim \psi^2\sigma '$ in $Z^2(Y,k)$. This means that 
$$\psi^2\sigma '=\del^0\psi^1\phi  *\del^2\psi^1\phi *\psi^2\sigma *\del^1\psi^1\phi^{-1} =\psi^2(\del^0\phi *\del^2\phi *\sigma *\del^1\phi^{-1} )$$
for some $\phi\in\Reg_H(X,k)$, and hence
$\sigma '=\del^0\phi *\del^2\phi *\sigma *\del^2\phi^{-1}$, where we used the fact that $\psi^1$ is an isomophism and $\psi^2$ is an injective algebra map. 
\end{proof}

Our aim here is to describe cocycle deformations of the bosonizations
$Y=X\# H$ of braided Hopf algebras $X$ in the category of crossed $H$-modules. The following calculation shows that equivalent cocycles lead to isomorphic deformations. 

\begin{Proposition} Let $Y=X\# H$ be the bosonization of a braided Hopf algebra $X$ in the category of crossed $H$-modules. If $\sigma ,\sigma '\in Z^2_H(Y,k)$ are in the same cohomology class then the cocycle deformations $Y_\sigma$ and $Y_{\sigma '}$ are isomorphic.
\end{Proposition}

\begin{proof} Suppose that $\sigma ' =\del^0\chi *\del^2\chi *\sigma *\del^1\chi^{-1} $ for some $\chi \in\Reg_H(X,k)$. It suffices to show that the equivariant coalgebra automorphism $\psi =\chi^{-1} *1*\chi :Y\to Y$ is actually also an algebra map $\psi :Y_{\sigma}\to Y_{\sigma '}$. And it is, since
\begin{eqnarray*}
\lefteqn{m_{\sigma '}(\psi x\ot\psi y) =\chi^{-1}(x_1)\chi^{-1}( y_1)m_{\sigma '}(x_2\ot y_2)\chi (x_3)\chi (y_3)} \\
&=&\chi^{-1}(x_1)\chi^{-1}(y_1)\sigma '(x_2\ot y_2)x_3y_3{\sigma '}^{-1}(x_4\ot y_4)\chi (x_5)\chi (y_5) \\
&=& \sigma (x_1\ot y_1)\chi^{-1}(x_2y_2)x_3y_3\chi (x_4y_4)\sigma^{-1}(x_5\ot y_5) \\
&=& \sigma (x_1\ot y_1)\psi (x_2y_2)\sigma^{-1}(x_3\ot y_3) \\
&=&\psi m_{\sigma}(x\ot y)
\end{eqnarray*}
implies that $\psi m_{\sigma '} =m_{\sigma}(\psi\ot\psi )$.
\end{proof}

\section{A 5-term sequence in \lq non-abelian' cohomology}

A commutative \lq pushout' square of (braided) Hopf algebras in the introduction and its bosonisation can help to get an explicit description of the deforming cocycles $\sigma$ on $B$ and of the corresponding cocycles on the bosonization $A=B\# H$ in terms of the $H$-invariant algebra maps $f\in \Alg_H(K,k)$, ,  Such squares of (braided) Hopf algebras
$$\begin{CD}
K @>\kappa >> R @. \qquad\rm{\ }\qquad  @. K\# H @>\kappa\# 1 >> R\# H \\
@V\ep VV  @V\pi VV    @.     @V\ep\# 1 VV  @V\pi\# 1 VV \\
k @>\iota >> B  @.         @.   H @>\iota\# 1>> B\# H
\end{CD}$$
induce a square of cosimplicial groups

$$\begin{array}{ccccccc}
\Reg_H(k,k) & ^{\del^0\atop{\longrightarrow}}_{\del^1\atop{\longrightarrow}} & \Reg_H (B,k) & ^{{\del^0\atop{\longrightarrow}}\atop{\del^1\atop{\longrightarrow}}}_{\del^2\atop{\longrightarrow}} & \Reg_H (B^2,k) & ^{{\del^0\atop{\longrightarrow}}\atop{\del^1\atop{\longrightarrow}}}_{{\del^2\atop{\longrightarrow}}\atop{\del^3\atop{\longrightarrow}}} & \Reg_H (B^3,k) \\
\|  & & \downarrow\pi^* & & \downarrow (\pi^2)^* & & \downarrow (\pi^3)^* \\
\Reg_H(k,k) & ^{\del^0\atop{\longrightarrow}}_{\del^1\atop{\longrightarrow}} &\Reg_H (R,k) & ^{{\del^0\atop{\longrightarrow}}\atop{\del^1\atop{\longrightarrow}}}_{\del^2\atop{\longrightarrow}} & \Reg_H (R^2,k) & ^{{\del^0\atop{\longrightarrow}}\atop{\del^1\atop{\longrightarrow}}}_{{\del^2\atop{\longrightarrow}}\atop{\del^3\atop{\longrightarrow}}} & \Reg_H (R^3,k) \\
\|  & & \downarrow\kappa^* & & \downarrow (\kappa^2)^* & & \downarrow (\kappa^3)^* \\
\Reg_H(k,k) & ^{\del^0\atop{\longrightarrow}}_{\del^1\atop{\longrightarrow}} & \Reg_H(K,k) & ^{{\del^0\atop{\longrightarrow}}\atop{\del^1\atop{\longrightarrow}}}_{\del^2\atop{\longrightarrow}} & \Reg_H (K^2,k) & ^{{\del^0\atop{\longrightarrow}}\atop{\del^1\atop{\longrightarrow}}}_{{\del^2\atop{\longrightarrow}}\atop{\del^3\atop{\longrightarrow}}} & \Reg_H (K^3,k)
\end{array}$$
where the trivial part has been omitted, and a similar square for the bosonisation. The natural injective group homomorphism $\psi :\Reg_H(X,k)\to\Reg_H(X\# H,k)$ induces a natural map between these squares. Here is a 5-term sequence for non-abelian cohomology in case $\kappa :K\to R$ has a $K$-bimodule coalgebra retraction $u:R\to K$.

\begin{Theorem} \label{5-term} If $\kappa K\to R$ has a $K$-bimodule coalgebra retraction then there is an exact sequences of pointed sets
\begin{eqnarray*}
1 \to \Alg_H(B,k) {\buildrel\pi^*\over\longrightarrow} \Alg_H(R,k) {\buildrel\kappa^*\over\longrightarrow} \Alg_H(K,k) {\buildrel\delta\over\longrightarrow} \H_H^2(B,k) {\buildrel\pi^*\over\longrightarrow} \H_H^2(R,k)
\end{eqnarray*}
and an injective map induced by the cosimplicial group homomorphism $\psi^*$ into a similar exact sequence involving the bosonisations. The connecting map $\delta :\Alg_H(K,k)\to \H_G^2(B,k)$ does not depend on the particular choice of the $K$-bimodule coalgebra retraction $u:R\to K$.
\end{Theorem}

\begin{proof} It is clear that $\pi^*:\Alg_H(B,k)\to\Alg_H(R,k)$ is injective and that $\kappa^*\pi^*=(\pi\kappa )^*=(\iota\ep )^*=\ep^*\iota^*$ is the trivial map. Moreover, if $\kappa^*(f)=\ep$ for $f\in\Alg_H(R,k)$ then, by the pushout property, there is a unique $f'\in\Alg_H(B,k)$ such that $\pi^*(f')=f$. To construct $\delta :\Alg_H(K,k)\to \H_H^2(B,k)$ observe first that
\begin{eqnarray*}
\Alg_H(K,k)&=&Z^1_H(K,k)=\H^1_H(K,k)=\{ f\in\Reg_H(K,k)|\del^1f=\del^2f*\del^0f\} \\
&=&\{ f\in\Reg_H(K,k)|\del^0f*\del^2f*\del^1fs=\ep\ot\ep\}.
\end{eqnarray*}
The existence of a $H$-invariant $K$-module coalgebra retraction $u:R\to K$ for the injection $\kappa :K\to R$
implies that, for every $f\in\Alg_H(K,k)$, the map $fu\in\Hom_H(R,k)$ is convolution invertible with inverse $fsu$. Then by Lemma \ref{nonab} the map
$$\sigma_R=\del u^*f=\del^0fu*\del^2fu*\del^1fsu :R\ot R\to k$$
is a convolution invertible 2-cocycle with inverse $\sigma_R^{-1}= \del^1fu*\del^2fsu *\del^0fsu$, in particular $\sigma_R(x\ot y)=fu(x_1)fu(y_1)fsu(x_2y_2)$. It satisfies the 2-cocycle conditions
$$\sigma_R(1\ot\iota )=\ep =\sigma_R(\iota\ot 1)) \ ,\ (\ep\ot\sigma_R)*\sigma_R(1\ot m) = (\sigma_R\ot\ep )*\sigma_R(m\ot 1).$$
Now $(\kappa\ot\kappa )^*\del^iu^*=\del^i\kappa^*u^*=\del^i$ for $i=0,1,2$,
so that $(\kappa\ot\kappa)^*\del fu =\del f=\ep\ot\ep$, since $f:K\to k$ is an algebra map. Moreover, because $u$ is a $H$-invariant $K$-bimodule coalgebra map and $f:K\to k$ is a H-invariant algebra map it follows that
$(fu\ot 1)c=(fu\ot 1)\tau $ and $fm_K=f\ot f$, so that
\begin{eqnarray*}
\del u^*f &=& (\ep\ot fu\ot fu\ot\ep\ot fsum)(\Delta_{R\ot R}\ot 1\ot 1)\Delta_{R\ot R} \\
&=&((fu\ot fu)c\ot fum)\Delta_{R\ot R}= (fu\ot fu\ot fsum)\Delta_{R\ot R} \\
&=& (fu\ot fu\ot fsum)(1\ot c\ot 1)(\Delta_R\ot\Delta_R) \\
&=&(fu\ot fu\ot fsum)(1\ot\tau\ot 1)(\Delta_R\ot\Delta_R)  
\end{eqnarray*}
and $\del fu(xr\ot r')=\ep (x)\del fu(r\ot r')=\del fu (r\ot r'x)$ for all $x\in K$ and $r, r'\in R$, which says that $\del fu:R\ot R\to k$ is a $K$-bimodule map. This means in particular that 
$$\del u^*f(K^+R\ot R + R\ot RK^+)=0$$
and hence that the cocycle $\sigma_R=\del u^*f:R\ot R\to k$ factors uniquely through $\pi\ot\pi :R\ot R\to B\ot B$, i.e: there exists a unique $\sigma :B\ot B\to k$ such that $(\pi\ot\pi )^*\sigma =\del u^*f$. Since $\pi :R\to B$ is a surjective Hopf algebra map, this $\sigma :B\ot B\to k$ is a 2-cocycle as well. So define
$$\delta :\Alg_H(K,k)\to Z_H^2(B,k)$$
by $\delta (f)=\sigma$.

Exactness at $\Alg_H(K,k)$: If $f\in\Alg_H(K,k)$ and $\delta f=\del\chi$ for some $\chi\in\Reg_H(B,k)$ then $\del fu=(\pi\ot\pi )^*\del\chi =\del\pi^*\chi$ and $g=\pi^*\chi^{-1}*fu\in\Reg_H(R,k)$ and $\kappa^*g=\kappa^*(\chi^{-1}\pi *fu)=\chi^{-1}\pi\kappa *f\kappa u =\chi^{-1}\iota\ep *f =\ep *f =f$. It remains to show that $g\in\Alg_H(R,k)$. But $\del g=\ep\ot\ep$, since
\begin{eqnarray*}
\del g&=&\del^0 g *\del^2g *\del^1g^{-1}\\
&=&\del^0(\chi^{-1}\pi *fu) )*\del^2(\chi^{-1}\pi *fu)*\del^1(fsu*\chi\pi ) \\
&=&\del^0\chi^{-1}\pi *\del^0fu*\del^2\chi^{-1}\pi *\del^2fu*\del^1fsu*\del^1\chi\pi \\
&=&\del^0\chi^{-1}\pi *\del^2\chi^{-1}\pi *\del^0fu *\del^2fu*\del^1fsu*\del^1\chi\pi \\
&=&\del^0\chi^{-1}\pi *\del^2\chi^{-1}\pi *\del fu*\del^1\chi\pi\\
&=&\del^0\chi^{-1}\pi *\del^2\chi^{-1}\pi *\del\chi\pi *\del^1\chi\pi\\
&=&\ep\ot\ep ,
\end{eqnarray*}
as $\del^0f'*\del^2f''=(f'\ot f'')c=(f'\ot f'')\tau =f''\ot f'=\del^2f''*\del^0f'$ for $f'\in\Reg_H(R,k)$, so that $g$ is an algebra map.

Conversely, if $f\in\Alg_H(R,k)$ then $\kappa^*f\in\Alg_H(K,k)$, $\del f\kappa u\in Z^2_H(R,k)$, $\delta f\kappa\in Z^2_H(B,k)$ and $(\pi\ot\pi )^*\delta\kappa^*(f)=\del f\kappa u$. Moreover,
\begin{eqnarray*}
(f\kappa u*fs)(r\kappa (x))&=&f\kappa u(r_1(r_2)_{-1}\kappa (x_1))fs((r_2)_0\kappa (x_2)) \\
&=&f\kappa u(r_1)f\kappa ((r_2)_{-1}x_1)fs((r_2)_0\kappa x_2) \\
&=&f\kappa u(r_1)f\kappa (x_1)fs\kappa (x_2)fs (r_2)\\
&=&\ep (x)(f\kappa u*f s )(r)
\end{eqnarray*}
for $r\in R$ and $x\in K$, in particular $(f\kappa u*fs)(RK^+)=0$. Hence, there is a unique $\chi\in\Reg_H(B,k)$ such that $f\kappa u*fs=\chi\pi$, and observe that $\chi$ is convolution invertible since $(f\kappa u*fs)^{-1}(RK^+)=(f*fs\kappa u)(RK^+)=0$ as well. This implies that $f\kappa u=\chi\pi *f$ and
\begin{eqnarray*}
\del f\kappa u&=&\del (\chi\pi *f )=\del^0(\chi\pi *f)*\del^2(\chi\pi *f)*\del^1(\chi\pi *f)^{-1} \\
&=&\del^0\chi\pi *\del^0f*\del^2\chi\pi *\del^2f*\del^1fs*\del^1\chi^{-1}\pi \\
&=&\del^0\chi\pi *\del^2\chi\pi *\del^0f*\del^2f*\del^1fs*\del^1\chi^{-1}\pi \\
&=&\del^0\chi\pi *\del^2\chi\pi *\del f*\del^1\chi^{-1}\pi\\ 
&=&\del\chi\pi,
\end{eqnarray*}
so that $(\pi\ot\pi )^*\delta f\kappa =\del f\kappa u=\del\pi^*\chi =(\pi\ot\pi )^*\del\chi$ and $\delta\kappa^* f=\delta f\kappa =\del\chi$, which is equvalent to $\ep\ot\ep $ under the equivalence relation on $Z^2_H(B,k)$. 

Exactness at $\H^2_H(B,k)$: If $f\in\Alg_H(K,k)$ then $(\pi\ot\pi )^*\delta f=\del fu$, which is equivalent to $\ep\ot\ep$ in $Z^2_H(R,k)$.

Conversely, if $\sigma\in Z^2_H(B,k)$ and $(\pi\ot\pi )^*\sigma =\del f$ for some $f\in\Reg_H(R,k)$ then $\del\kappa^*f=(\kappa\ot\kappa )^*\del f=(\pi\kappa\ot\pi\kappa )^*\sigma =\ep\ot\ep$, so that $\kappa^*f=f\kappa\in\Alg_H(K,k)$, $\del f\kappa u\in Z^2_H(R,k)$ and $\delta (f\kappa )\in Z^2_H(B,k)$. It suffices to prove that $\delta f\kappa$ is equivalent to $\sigma$ in $Z_H^2(B,k)$. Now, since $\del f(RK^+\ot R+R\ot RK^+)=0$ it follows that $\del f(r\ot\kappa (x))=\ep (x)\del f(r\ot 1)+\del f(r\ot (\kappa (x)-\ep(x)))=\ep (x)\del f(r\ot 1)=(\ep\ot\ep )(r\ot\kappa(x))$, which implies that
$$f(r\kappa (x))=\del^1f(r\ot\kappa (x))=\del^2f*\del^0f(r\ot\kappa (x))=f(r)f\kappa (x)$$
for all $r\in R$ and $x\in K$. Then $(f\kappa su*f)(RK^+)=0$, since
\begin{eqnarray*}
(f*f\kappa su)(r\kappa (x))&=&f(r_1(r_2)_{-1}\kappa (x_1))f(\kappa su(r_2)_0\kappa (x_2)) \\
&=&f(r_1)f((r_2)_{-1}\kappa (x_1))f\kappa s(x_2)f\kappa su((r_2)_0) \\
&=&f(r_1)f\kappa (x_1)f\kappa s(x_2)f\kappa su(r_2)\\
&=&\ep (x)(f*f\kappa su)(r),
\end{eqnarray*}
and hence there is a unique $\chi\in\Reg_H(B,k)$ such that $f*f\kappa su=\pi^*\chi$, that is $f=\chi\pi *f\kappa u$. Then
\begin{eqnarray*}
(\pi\ot\pi )^*\sigma &=&\del f=\del^0(\chi\pi *f\kappa u)*\del^2(\chi\pi *f\kappa u)\del^1(\chi\pi *f\kappa u)^{-1} \\
&=&\del^0\chi\pi *\del^0f\kappa u*\del^2\chi\pi *\del^2f\kappa u*\del^1f\kappa su*\del^1\chi^{-1}\pi \\
&=&\del^0\chi\pi *\del^2\chi\pi *\del^0f\kappa u*\del^2f\kappa u*\del^1f\kappa su*\del^1\chi^{-1}\pi \\
&=&\del^0\chi\pi *\del^2\chi\pi *\del f\kappa u*\del^1\chi^{-1}\pi \\
&=&(\pi\ot\pi )^*(\del^0\chi *\del^2\chi *\delta f\kappa *\del^1\chi^{-1} ),
\end{eqnarray*}
since $\del^0f\kappa u*\del^2\chi\pi =\del^2\chi\pi *\del^0f\kappa u$, and
thus
$$\sigma =\del^0\chi *\del^2\chi *\delta f\kappa *\del^1\chi^{-1} ,$$
so that $\sigma$ is equivalent to $\delta f\kappa $ in $Z^2_H(B,k)$. The remaining assertions are now obvious.

Similar and somewhat simpler arguments lead to an exact sequence of pointed sets
\begin{eqnarray*}
\H^1_H(B\# H,k) ^{\pi^*\atop{\longrightarrow}} \H^1_H(R\# H,k) ^{\kappa^*\atop{\longrightarrow}} \H^1_H(K\# H,k) ^{\delta\atop{\longrightarrow}} \H^2_H(B\# H,k) ^{\pi^*\atop{\longrightarrow}} \H^2_H(R\# H,k)
\end{eqnarray*}
for the bosonisations, and the map $\psi^*$ of cosimplicial groups induces an injective map between the two sequences. As an alternative, given the exact sequence for the bosonisations and the map $\psi^*$ the sequence for the braided square also follows directly.

It remains to show that any two $K$-bimodule coalgebra retractions $u, u' :K\to R$ lead to the same connecting map $\delta :\Alg_H(K,k)\to \H_H^2(B,k)$. Observe that $\ker\pi =K^+R+RK^+$. For $f\in\Alg_H(K,k)$ let $\sigma ,\sigma '\in Z_H^2(B,k)$ be such that $(\pi\ot\pi )^*\sigma =\del fu$ and  $(\pi\ot\pi )^*\sigma ' =\del fu'$. If $x\in K^+$ and $r\in R$ then $\Delta (xr)=x_1(x_2)_{-1}r_r\ot (x_2)_0r_2$ and
\begin{eqnarray*}
fu'*fsu(xr) & = & fu'(x_1(x_2)_{-1}r_1)fsu((x_2)_0r_2) \\
& = & f(x_1)f((x_2)_{-1}u(r_1))fs((x_2)_0)fsu(r_2) \\
& = & \ep (x)fu'(r_1)fsu(r_2)=0
\end{eqnarray*}
and a similar argument shows that $fu' *fsu(rx)=0$, so that $\chi =fu' *fsu\in\Reg_H(B,k)$. Moreover, since the faces $\del^i :\Reg_H(R,k)\to \Reg_H(R\ot R,k)$ are group homomorphisms and since $\del^0f' *\del^2f'' =\del^2f'' *\del^0f'$ it follows that
\begin{eqnarray*}{l}
\lefteqn{(\pi\ot\pi )^*(\del^0\chi *\del^2\chi *\sigma *\del^1\chi^{-1})}  \\
& =  & \del^0(fu'*fsu)*\del^2(fu'*fsu)*\del fu *\del^1(fu'*fsu ) \\
& =  & \del^0fu'*\del^2fu' *\del^1fsu' =(\pi\ot\pi )^*\sigma ' .
\end{eqnarray*}
But $(\pi\ot\pi )^*:\Reg_H(B\ot B,k) to\Reg_H(R\ot R,k)$ is injective, so that $\del^0\chi *\del^2\chi *\sigma *\del^1\chi^{-1} =\sigma '$, which means that $\sigma$ and  $\sigma '$ are in the same cohomology class.
\end{proof}

\noindent
{\bf Remark.} For Hochschild cohomology, which in some cases can be viewed as the infinitesimal part of the \lq multiplicative'  cohomology, such a sequence  (now of vector spaces) also exists \cite{GM}. The proofs are similar but somewhat simpler in that case, and the requirement that the retraction $u:R\to K$ be a coalgebra map is not needed.

\section{Applications to the lifting process}

Every lifting of a given diagram of special finite Cartan type is by \cite{GM} a cocycle deformation of the bosonisation $B(V)\# kG$ of the Nichols algebra $B(V)$ and is completely determined by a $G$-invariant algebra map $f:K(V)\to k$. In the presence of a $K(V)$-module coalgebra retraction $u:R(V)\to K(V)$ for the injection $\kappa :K(V)\to R(V)$ the deforming cocycle can be determined via the connecting map $\delta :\Alg_G(K,k)\to \H_G^2(R,k)$ described in the last section.
Observe that in our case, $\Alg_G(B,k)=\Alg_G(R,k)=\{ \ep\}$, and that $\delta$ is injective. The simple root vectors $x_{\alpha}$, where $\alpha\in\Phi^+$ is a simple root, generate $R$ as an algebra.  Moreover,  $f(x_{\alpha})=f(gx_{\alpha})=\chi_{\alpha}(g)f(x_{\alpha})$ for every $g\in G$ and $f\in\Alg_G(R,k)$. It follows that $\Alg_G(R,k)=\{ \ep \}$, since $q_{\alpha}=\chi_{\alpha}(g_{\alpha})$ is a non-trivial root of unity for every simple root $\alpha$.

By \cite{GM} Theorem 2.2 (\cite{AS}, Theorem 2.6) it follows that the map $\vartheta :R\to B\ot K$, given by $\vartheta (x^az^{a'})=x^a\ot z^{a'}$, is a $K$-module isomorphism. The $K$-bimodule retraction
$$u=(\ep\ot 1)\vartheta :R\to K$$
for the injection $\kappa :K\to R$ has kernel $B^+R$ and is a $K$-bimodule map.

\subsection{Type $A_1$} In this case the retraction $u:R\to K$ is a $K$-module coalgebra map, since the obvious injection $v:B\to R$ is a coalgebra map, so that $B^+R$ is a coideal in $R$. The injective map
$$\delta :\Alg_G(K,k)\to\H_G^2(B,k)$$
is given by $\sigma =\delta f=(fu\ot fu)*fsum (v\ot v)=fsum(v\ot v)$, that is $\sigma (x^i\ot x^j)=fsu(x^{i+j})$ and $\sigma^{-1}(x^i\ot x^j)= fu(x^{i+j})$ for $0\le i,j <N$. Using
$$\Delta (x^m\ot x^n)=\sum_{0\le i\le m; 0\le j\le n}{m\choose i}_q{n\choose j}_qx^ig^{m-i}\ot x^jg^{n-j}\ot x^{m-i}\ot x^{n-j}$$ and the identity
$$\sum_{i+j=r}{\m\choose i}_q{n\choose j}_qq^{j(m-i)}={{m+n}\choose r}_q=1$$
of  \cite{Ka} it follows that
$$m_{\sigma}(x^m\ot x^n) = \left \{ \begin{array}{ll}
x^{m+n}, & \mbox{if $m+n<N$} \\
fs(z)x^{m+n-N}(1-g^N), & \mbox{if $m+n\ge N$}
\end{array} \right . $$

\subsection{Quantum planes} The general quantum plane $V=kx_1\oplus kx_2$ has $G$-coaction $\delta (x_i)=g_i\ot x_i$ and $G$-action $gx_i=\chi_i(g)x_i$, where $\chi_1(g_2)\chi_2(g_1)=1$ and $q=\chi_1(g_1)$ is a primitive root of unity of order $N$. Moreover, $\chi_i^N=\ep =\chi_1\chi_2$, so that
$\chi_1(g_i)=q$ and $\chi_2(g_i)=q^{-1}$. In the free Hopf algebra $k<x_1, x_2>$ the relation $x_2x_1=qx_1x_2+z_{21}$, where $z_{21}=[x_2, x_1]=x_2x_1-qx_1x_2$, can be used to construct a PBW-basis. The following Lemma, which will also be used later, is helpful in this connection.

\begin{Lemma} \label{link} For a quantum plane with linkable vertices, i.e: with $\chi_1\chi_2=\ep$, the relations
$$x_2^mx_1^n =\sum_{r=0}^lq^{(m-r)(n-r)}r!_q{m\choose r}_q{n\choose r}_qx_1^{n-r}x_2^{m-r}z_{21}^r +p_{mn}$$
hold in $k<x_1, x_2>$, where $l=\min \{ m,n\}$ and $p_{mn}$ is an element in the ideal generated by $[x_1,z_{21}]$ and $[x_2,z_{21}]$.
\end{Lemma}

\begin{proof} Since the vertices are linkable, we have $z_{21}x_i=x_iz_{21}-[x_i,z_{21}]$. It follows by induction on $m$ that
$$x_2^mx_1=q^mx_1x_2^m+m_qx_2^{m-1}z_{21} -\sum_{i=1}^{m-1}q^ix_2^{m-1-i}[x_2^i,z_{21}]$$
where $[x_2^i,z_{21}]=\sum_{k=1}^{i-1}x^{i-k}[x_2,z_{21}]x_2^{k-1}$, and then, if $m\ge n$, by induction on $n$
$$x_2^mx_1^n =\sum_{r=0}^lq^{(m-r)(n-r)}r!_q{m\choose r}_q{n\choose r}_qx_1^{n-r}x_2^{m-r}z_{21}^r +p_{mn},$$
where $p_{m(n+1)}=p_{mn}x_1+\sum_{r=0}^nq^{(m-r)(n-r)}r!_q{m\choose r}_q{n\choose r}_q(x_1^{n-r}x_2^{m-r}[x_1,z_{21}^r] +p_{(m-r)1}z_{21}^r)$ and $p_{m1}=\sum_{i=1}^{m-1}q^ix_2^{m-i}[x_2^i,z_{21}]$. Here we used the identities ${m\choose {r-1}}{{(m-r+1)_q}\over{r_q}}={m\choose r}_q$ and ${n\choose r}_q+q^{n+1-r}{n\choose{r-1}}_q={{n+1}\choose r}_q$.

On the other hand,  by induction on $n$ we get
$$x_2x_1^n=q^nx_1^nx_2+n_qx_1^{n-1}z_{21} -\sum_{i=1}^{n-1}(n-i)_qx_1^{n-i-1}[x_1,z_{21}]x_1^{i-1}$$
and then, if $m\le n$, by induction on $m$
$$x_2^mx_1^n=\sum_{r=0}^mq^{(m-r)(n-r)}r!_q{m\choose r}_q{n\choose r}_qx_1^{n-r}x_2^{m-r}z_{21}^r -{p'}_{mn}$$
with ${p'}_{(m+1)n}=x_2{p'}_{mn}+\sum_{r=0}^mq^{(m-r)(n-r)}r!_q{m\choose r}{n\choose r}\big( (n-r)_qx_1^{n-r-1}[x_2^{m-r},z_{21}]+{p'}_{1(n-r)}x_2^{m-r}\big) z_{21}^r$ and ${p'}_{1n}=\sum_{i=1}^{n-1}(n-i)_qx_1^{n-i-1}[x_1,z_{21}]x_1^{i-1}$. Here the identities ${n\choose{r-1}}_q{{(n+1-r)_q}\over{r_q}}={n\choose r}_q$ and ${m\choose r}_q+q^{m+1-r}{m\choose{r-1}}_q={{m+1}\choose r}_q$ were used.
\end{proof}

The elements  $x_i^N=z_i$ and $[x_2,x_1]=z_{21}$ are primitive in $k<x_1, x_2>$, and so are $[x_1, z_2]$, 
$[x_2,z_1]$ and $[z_i,z_{21}]$ for
$i=1, 2$. The ideal  generated by the  elements $[x_1, z_2]$, $[x_2,z_1]$, $[z_1, z_{21}]$ and $[z_2, z_{21}]$ in the braided Hopf algebra $k<x_1, x_2>$ is therefore a Hopf ideal, so that
$$R=k<x_1, x_2>/( [x_1, z_2], [x_2,z_1], [z_1, z_{21}], [z_2, z_{21}])$$
is a Hopf algebra in the category of crossed $kG$-modules. It follows from the Lemma above that $[z_2, z_1]=z_2z_1-z_1z_2=0$. Thus, if  $K$ is the Hopf subalgebra of $R$ generated by
$z_1$, $z_2$ and $z_{21}$, then $K=k[z_1, z_2, z_{21}]$ as an algebra, and
$$\begin{CD}
K @>\kappa >> R \\
@V\ep VV  @V\pi VV \\
k @>\iota >>  B
\end{CD}$$
is  a pushout square of braided Hopf algebras, where $B=k<x_1, x_2>/(z_1, z_2, z_{21})$ is the Nichols algebra of the quantum plane. By the Lemma above $R\cong (B\ot K)\oplus J$ as a vector space, where $J$ is the ideal in $R$ generated by $[x_1, z_{21}]$ and $[x_2, z_{21}]$, which is not a Hopf ideal.

\begin{Proposition} For the quantum plane the injection $\kappa :K\to R$ has a $K$-bimodule coalgebra retraction $u :R\to K$ defined by  $u(x^az^b + J)=\ep (x^a)z^b$.
\end{Proposition}

\begin{proof} It is clear that the lineaer map $u:R\to K$ defined by $u(x^az^b+J)=\ep (x^a)z^b$ satisfies $u\kappa = 1_K$. It is a $K$-bimodule map, since in $R$ we have $[x_i, z_j]=0$ and $[x_i,z_{21}]\in J$. It is also a coalgebra map, since its kernel $\ker u =(B^+\ot K)\oplus J$ is a coideal.
\end{proof}

Theorem \ref{5-term} is therefore applicable and, since $\Alg_G(R,k)=\{ \ep \}$, it follows that the connecting map $\delta :\Alg_G(K,k)\to \H^2_G(B,k)$ is injective (see proof of Proposition \ref{injective}). It is determined by $(\pi\ot\pi )^*\delta f=\del (fu)=(fu\ot fu)*fsum_R$ and, since the obvious injection $v:B\to R$ is a coalgebra map, we see that
$$\sigma (x^a\ot x^b)=\del (fu)(x^a\ot x^b)=fsu (x^ax^b)$$
for $0\le a_i, b_i <N$, where the cocycle $\sigma =\del (fu)(v\ot v)$ represents the cohomology class $\delta f\in\H_G^2(B,k)$. In particular, in view of the definition of $u$ and Lemma \ref{link},
$$\sigma (x_i^m\ot x_j^n)=\left\{ \begin{array}{ll}
\delta^i_j\delta^{m+n}_Nfs(z_i)) & \mbox{, if $i\le j$} \\
\delta^m_nn!_qfs(z_{21}^n) & \mbox{, if $i=2> j=1$}
\end{array} \right . $$
for $0\le m,n <N$.
Here is the connection to Hochschild cohomology, a result  which is also applicable in a more general context. Recall first that by  \cite{GM} there is a Kunneth type isomorphism in equivariant Hochschild cohomology
$$H^2_G(B,k)\cong H^2_G(B_1,k)\oplus H^2_G(B_2,k)\oplus (H^1(B_1,k)\ot H^1(B_2,k))_G,$$
where $B_i=k[x_i]/(x_i^N)$ are Nichols algebras of quantum lines and $H^1(B_i,k)=\Der (B_i,k)\cong\Hom ({B_i^+}/{(B_i^+)^2}, k)$. This means that every $\zeta\in H_G^2(B,k)$ has a unique decomposition of the form $\zeta =\zeta_1+\zeta_2+\zeta_{21}$. The following result has also been obtained recently with somewhat different methods  in \cite{ABM}, section 5. 

\begin{Theorem} \label{q-plane} For any quantum plane the diagram
$$\begin{CD}
\Der_G(K,k) @>\delta_{Hoch} >> H^2_G(B,k) \\
@V\exp VV  @V\Exp_qVV \\
\Alg_G(K,k) @>\delta >> \H^2_G(B,k)
\end{CD}$$
commutes if $\exp (d )=e^d$ and  $\Exp_q(\zeta )=e_q^{\zeta_1}*e_q^{\zeta_2}*e_q^{\zeta_{21}}$, where $e^d=\sum_{n\ge 0}{{d^n}\over{n!}}$ and  $\exp_q (\xi )=e_q^{\xi} =\sum_{n\ge 0}{{\xi^n}\over {n!_q}}$ are the convolution exponential and $q$-exponential, respectively.
\end{Theorem}

\begin{proof} It is clear that $\exp :\Der_G(K,k)\to\Alg_G(K,k)$, given by the convolution power series $\exp (d)=e^d=\sum_{n\ge 0}{{d^n}\over{n!}}$, is an isomorphism of abelian groups, since the Hopf algebra $K=k[z_1, z_2, z_{21}]$ is a polynomial algebra. By \cite{GM} there is a Kunneth type isomorphism in equivariant Hochschild cohomology
$$H^2_G(B,k)\cong H^2_G(B_1,k)\oplus H^2_G(B_2,k)\oplus (H^1(B_1,k)\ot H^1(B_2,k))_G,$$
where $B_i=k[x_i]/(x_i^N)$ are Nichols algebras of quantum lines and $H^1(B_i,k)=\Der (B_i,k)\cong\Hom ({B_i^+}/{(B_i^+)^2}, k)$. The connecting map $\delta_{Hoch}:\Der_G(K,k)\to H_G^2(B,k)$ is an isomorphism, since $\Der_G(R,k)=0$ and simce $\dim \Der_G(K,k)=3=\dim H_G^2(B,k)$. The connecting map $\delta :\Alg_G(K,k)\to \H_G^2(B,k)$ is injective, as mentioned above, since $\Alg_G(R,k)=\{ \ep \}$. Moreover, every element $d\in\Der_G(K,k)$ has a unique expression of the form $d=d_1+d_2+d_{21}$, every $f\in\Alg_G(K,k)$ is uniquely of the form $f_1*f_2*f_{21}$, where the notation is self explanatory, and $e^d=e^{d_1}*e^{d_2}*e^{d_{21}}$. For a general $f=f_1*f_2*f_{21}\in\Alg_G(K,k)$ one obtains the formula
\begin{eqnarray*}
\delta f(x_1^kx_2^m\ot x_1^nx_2^l) & = & \delta^k_0\delta^l_0\delta^m_nn!_qfs(z_{21})^n \\
& + & \delta^k_0\delta^{m+l-N}_nn!_q{m\choose n}_qfs(z_2)fs(z_{21})^n \\
& + & \delta^l_0\delta^m_{k+n-N}m!_q{n\choose m}_qfs(z_1)fs(z_{21})^m \\
& + & \delta^{m+l}_{k+n}\delta^{k+n-N}_rq^{(N-l)(N-k)}r!_q{m\choose r}_q{n\choose r}_qfs(z_1)fs(z_2)fs(z_{21})^r \\
& = & \delta f_1*\delta f_2*\delta f_{21}(x_1^kx_2^m\ot x_1^nx_2^l),
\end{eqnarray*}
where $\delta f=\delta f_1*\delta f_2*\delta f_{21}$ also follows directly from the fact that the cofaces $\del^i:\Reg_G(R,k)\to\Reg_G(R\ot R,k)$ are algebra maps and that $\del (fu)(v\ot v)=\del^1(fu)(v\ot v)$. This formula shows in particular that
\begin{eqnarray*}
\delta e^{d_1}(x_1^kx_2^m\ot x_1^nx_2^l) & = & \delta^m_0\delta^l_0d_1s(z_1) \\
\delta e^{d_2}(x_1^kx_2^m\ot x_1^nx_2^l) & = & \delta^k_0\delta^n_0d_2s(z_2) \\
\delta e^{d_{21}}(x_1^kx_2^m\ot x_1^nx_2^l) & = & \delta^k_0\delta^l_0\delta^m_n n!_qd_{21}s(z_{21})^n
\end{eqnarray*}
On the other hand, drawing on Lemma \ref{link} again, for $d=d_1, d_2, d_{21}$ compute
$$e_q^{\delta_{Hoch}d}=e_q^{dsum}=\sum_{t\ge 0}{{(dsum)^t}\over{t!_q}}$$
by evaluating the convolution powers $(dsum)^t(x_1^kx_2^m\ot x_1^nx_2^l)$ for $t>0$ to get
\begin{eqnarray*}
(d_1sum)^t(x_1^kx_2^m\ot x_1^nx_2^l) & = & \delta^s_1\delta^m_0\delta^l_0\delta^{k+n}_Nd_1s(z_1) \\
(d_2sum)^t(x_1^kx_2^m\ot x_1^nx_2^l) & = &\delta^s_1 \delta^k_0\delta^n_0\delta^{m+l}_Nd_2s(z_2) \\
(d_{21}sum)^t(x_1^kx_2^m\ot x_1^nx_2^l) & = & \delta^k_0\delta^l_0\delta^m_n\delta^t_n(n!_q)^2(d_{21}s(z_{21}))^n
\end{eqnarray*}
and therefore $e_q^{\delta_{Hoch}d}=\delta e^d$ for the specified derivations. This means that the map $\Exp_q:H^2_G(B,k)\to \H^2_G(B,k)$ is given by $\Exp_q(\zeta)=e_q^{\zeta_1}*e_q^{\zeta_2}*e_q^{\zeta_{21}}$.
\end{proof}

\noindent
{\bf Remark:} Observe that in general $\delta f_1$ and $\delta f_2$ do not commute with $\delta f_{21}$, since for example $\delta f_2*\delta f_{21}(x_2^2\ot x_1x_2^{N-1})=q^{-1}(1+q)f_2(z_2)f_{21}(z_{21})$ and $\delta f_{21}*\delta f_2(x_2^2\ot x_1x_2^{N-1})=(1+q)f_{21}(z_{21})f_2(z_2)$, so that in general $Exp_q(\zeta )\ne e_q^{\zeta}$ (see also \cite{ABM}). But, if
$\zeta_{21}=0$, that is $f_{21}=\ep $, then $\Exp_q(\zeta )=e_q^{\zeta}$, since $\delta f_2*\delta f_1=\delta f_1*\delta f_2$, a result already obtained in \cite{GM}.

\subsection{Linking} Let $V=kx_1\oplus\ldots\oplus kx_{\theta}$ be any special diagram of finite Cartan type, and suppose that $i<j$ is a linkable pair, i.e: $\chi_i\chi_j=\ep $. Then $i$ and $j$ are in different components of the Dynkin diagram, and they are not linkable to any other vertices. Let $B=TV/I$ be the Nichols albebra of $V$, where $I$ is the ideal generated by the usual set $S$. If $S_{ij}=S\setminus \{ z_{ji}\}$, where $z_{ji}=[x_j,x_i]$, then the ideal $I_{ij}$ in $TV$ generated by $S_{ij}$ is still a Hopf ideal and $R_{ij}=TV/I_{ij}$ is a braided Hopf algebra. The kernel of the canonical projection $\pi :R_{ij}\to B$ is the ideal generated by $z_{ji}$, which is a Hopf ideal, since $z_{ji}$ is primitive.  If $K_{ij}$ is the Hopf subalgebra of $R_{ij}$ generated by $z_{ji}$ then
$$\begin{CD}
K_{ij} @>\kappa >> R_{ij} \\
@V\ep VV  @V\pi VV \\
k @>\iota >> B
\end{CD}$$
is a pushout square. Moreover, as a vector space $R_{ij}\cong (B\ot K_{ij})\oplus J{ij}$, where $J_{ij}$ is the ideal generated by the set $\{ [x_k,z_{ji}] |1\le k\le \theta\}$, which is not a Hopf ideal.

\begin{Proposition} For any special diagram of finite Cartan type and any linkable pair of verices $i<j$ in its Dykin diagram, the linear map $u:R_{ij}\to K_{ij}$, given by $u(x^a\ot z_{ji}^n+J_{ij})=\ep (x^a)z_{ij}^n$, is a  $K_{ij}$-bimodule coalgebra retraction for the inclusion $\kappa :K_{ij}\to R_{ij}$.
\end{Proposition}

\begin{proof} It is clear that the  $u:R_{ij}\to K_{ij}$ just defined is a linear map satisfying $u\kappa =1_{K_{ij}}$. It is a $K_{ij}$-bimodule map, since in $R_{ij}$ the element $[x^a, z_{ji}]$ is in $J_{ij}$ for every $x^a\in B$. It is a coalgebra map, since $(B^+\ot K_{ij})\oplus J_{ij}$is a coideal in $R_{ij}$.
\end{proof}

Our Theorem \ref{link} and the corresponding result for Hochschild cohomology are therefore applicable. Since, as an algebra, $R_{ij}$ is generated by the set  $\{ x_l| 1\le l\le\theta \}$, and since  the $\chi_l(x_l)$ are non-trivial roots of unity, we conclude that $\Der_G(R_{ij},k)=0$ and $\Alg_G(R_{ij},k)= \{ \ep \}$. Moreover, for the polynomial Hopf algebra $K_{ij}=k[z_{ji}]$, the convolutian exponential map $\exp :\Der_G(K_{ij},k)\to \Alg_G(K_{ij},k)$ is an isomorphism of groups, and the diagram
$$\begin{CD}
\Der_G(k_{ij},k) @>\delta_{Hoch} >>  H^2_G(B,k) \\
@V\exp VV  @. \\
\Alg_G(K_{ij},k) @>\delta >> \H^2_G(B,k)
\end{CD}$$
carries some information. In this generality there is no obvious map relating $H^2_G(B,k)$ to $\H^2_G(B,k)$, but the diagram relates the image of $\delta_{Hoch}$ to $\H^2_G(B,k)$. More precisely, by the Kunneth formula for the equivariant Hochschild cohomology of Nichols algebras, $\im \delta_{Hoch}\subseteq (\Der (B_i,k)\ot\Der (B_j,k))_G$, where $B_i$ and $B_j$ are the Nichols algebras of the components of the Dynkin diagram containing the vertices $i$ and $j$, respectively. 

\begin{Corollary} Let $i<j$ be a linkable pair of vertices in a special diagram of finite Cartan type. For the derivation $d\in \Der_G(K_{ij},k)$ the Hochschild cocycle representing $\zeta =\delta_{Hoch}d\in H^2_G(B,k)$ is given by $\zeta (x^a\ot x^b)=\delta^a_{e_j}\delta^b_{e_i}d(z_{21})$ and $e_q^{\zeta }=\delta e^d\in \H^2_G(B,k)$
\end{Corollary}

\begin{proof} Replacing the pair $(1,2)$ by $(i,j)$, Lemma \ref{link} holds for any linkable pair $i<j$ in any special diagram of finite Cartan type. It shows that the Hochschild cocycle $\zeta =\delta_{Hoch}$ is of the form specified. Together with arguments, similar to those used in Theorem \ref{q-plane}, it also shows that $\delta e^d=e_q^\zeta$.
\end{proof}

\subsection{Type $A_1\times\ldots\times A_1$} The general quantum linear space $V=kx_1\oplus\ldots\oplus kx_{\theta}$ of dimension $\theta $ has $G$-coaction $\delta (x_i)=g_i\ot x_i$ and $G$-action $gx_i=\chi_i(g)x_i$, where $\chi_i^{N_i}=\ep$ and $\chi_i(g_j)\chi_j(g_i)=1$ for $i\ne j$.  

A vertex $i$ is linkable to at most one other vertex, since the order $N_i$ of $q_{ii}=\chi_i(g_i)$ is supposed to be greater than $2$. The vertex set $\{ 1, 2, \dots ,\theta \}$ can therefore be decomposed into a set $L$ of linkable pairs of the form $i<j$ and a set of non-linkable singletons $L^{\perp}$, and it can be ordered accordingly. A quantum linear space is therefore a collection of quantum planes together with a bunch of quantum lines with pushout squares
$$\begin{CD}
K_{ij} @>\kappa_{ij}>> R_{ij} @. \quad  , \quad @. K_l @>\kappa_l>> R_l \\
@V\ep VV  @V\pi_{ij}VV  @.   @V\ep VV  @V\pi_lVV \\
k @>\iota >> B_{ij}  @. \quad  \ \quad  @.  k @>\iota >> B_i
\end{CD}$$
for $(i,j)\in L$ and $l\in L^{\perp}$, respectively. The braided tensor product of all these squares represents the Nichols algebra of the quantum linear space.  The following considerations about such
braided tensor products together with the results for $\theta\le 2$ will describe the deforming cocycles for all quantum linear spaces.

If  a subset $S$ of $\{ 1,2,\ldots ,\theta\}$ is such that none of its vertices is linkable to any vertex not in $S$ then
the complement $T$ has the same property and $\{ 1,2,\ldots ,\theta\}=S\cup T$. The elements of $K_S$ commute with the elements of $R_T$ and the elements of $K_T$ commute with those of $R_S$. It follows that there is a commutative diagram of coalgebras
$$\begin{CD}
K @>\kappa >> R @>\pi >> B \\
@V\rho_KVV @V\rho_RVV @V\rho_BVV \\
K_S\ot K_T @>\kappa_S\ot\kappa_T>> R_S\ot R_T  @>\pi_S\ot\pi_T>> B_S\ot B_T
\end{CD}$$
with $\rho =(p_S\ot p_T)\Delta$ is an isomorphism with inverse $\rho^{-1}=m(i_S\ot i_T)$. The projections $e_S=i_Sp_s$ and $e_T=i_Tp_T$ on $K$, $R$ and $B$ have the property that
$$e_S*e_T=\rho^{-1}\rho=1 \ ,\  ue_S=e_Su \ ,\  ue_T=e_Tu \ ,\ u=e_Su*e_Tu=ue_S*ue_T$$
and, moreover, since the elments of $K_S$ commute with those of $K_T$, also $e_T*e_S=e_S*e_T=1_K$ on $K$. The latter is of course not true on$R$ and $B$, because $e_T*e_S(x_S^ax_T^b)=\chi^a(g^b)e_S*e_T(x^ax^b$.  With $u_S=p_Sui_S$ and $u_T=p_Tui_T$ the diagram
$$\begin{CD}
R @>p_S>> R_S @>i_S>> R @<<i_T< R_T @<<p_T< R \\
@VuVV          @Vu_SVV     @VuVV          @Vu_TVV    @VuVV \\
K  @>p_S>> K_S @>i_S>> K @<<i_T<  K_T @<<p_T< K
\end{CD}$$
commutes.  The projections $e_S=i_S\pi_S$ and $e_T=i_T\pi_T$ on $R$ and $K$ satisfy
$$e_Su=ue_S \  ,\  e_Tu=ue_T \  ,\  e_S*e_T=\rho_S^{-1}\rho_S=1 \  , \  u=ue_S*ue_T=e_Su*e_Tu$$
and the diagram
$$\begin{CD}
R @>u>> K \\
@V\rho_R VV  @V\rho_K VV \\
R_S\ot R_T @>u_S\ot u_T>> K_S\ot K_T
\end{CD}$$
commutes.
Moreover, since the elements of $K_S$ and $K_T$ commute, we have $e_T*e_S=e_S*e_T= 1_K$ on $K$.
This is of course not the case on $R$ or on $B$, because $e_T*e_S(x^ax^{a'}z^bz^{b'})=\chi^a(g^{a'})e_S*e_T(x^ax^{a'}z^bz^{b'})$.

\begin{Proposition} \label{tp} Suppose that $\{ 1,2,\ldots ,\theta\} =S\cup T$ is such that none of the vertices of $S$ is linkable to any vertex of $T$, then $u=u_S*u_T:R\to K$, where $u_S=ue_S=e_Su$ and $u_T=ue_T=e_Tu$
and the square
$$\begin{CD}
\Alg_G(K_S,k)\times\Alg_G(K_T,k) @>\delta >> \H_G^2(B_S,k)\times\H_G^2(B_T,k) \\
@V\rho^1VV  @V\rho^2VV \\
\Alg_G(K,k)  @>\delta >>  \H_G^2(B,k)
\end{CD}$$
commutes, where $\rho^1(f,f')=(f\ot f')\rho$ and $\rho^2(\sigma ,\sigma ')=(\sigma\ot\sigma ')(1\ot c\ot 1)(\rho\ot\rho )$. Moreover, $\rho^1$ is an isomorphism, while $\rho^2$ is injective.
\end{Proposition}

\begin{proof} With our assumptions and $f\in\Alg_G(K,k)$ we have $f=f(e_S*e_T)=fe_S*fe_T$ and$fu=fe_Su*fe_Tu$. Observe that the inverse of $\rho^1$ is given by $(\rho^1)^{-1}(f)=(fi_S,fi_T)$:
$$\rho^1(\rho^1)^{-1}(f)=(fi_S\ot fi_T)\rho =f(e_S*e_T)=f,$$
while
$$(\rho^1)^{-1}\rho^1(f_S,f_T)=((f_S\ot f_T)\rho i_S, (f_S\ot f_T)\rho i_T)=(f_S\ot \ep )\Delta , \ep\ot f_T )\Delta )=(f_S, f_T),$$
since $\rho i_S=(1\ot\iota\ep )\Delta$ and $\rho i_T=(\iota\ep\ot 1)\Delta$. A similar argument shows that 
$\rho^2$ has a left inverse $\psi$ given by $\psi (\sigma )=(\sigma (i_S\ot i_S), \sigma (i_T\ot i_T))$.
 
The diagram commutes, because
\begin{eqnarray*}
\del^i(\rho^1(f_S,f_T)u) & = & \del^i((f_S\ot f_T)\rho u)=\del^i(f_Su_Sp_S*f_Tu_Tp_T) \\
& = & \del^i(f_Su_Sp_S)*\del^i(f_Tu_Tp_T)=(\del^if_Su_S\ot\del^if_Tu_T)\rho_{R\ot R} \\
& = & \theta^2(\del^if_Su_S,\del^if_Tu_T),
\end{eqnarray*}
where we used $d^i(p_S\ot p_S)=p_Sd^i$, $d^i(p_T\ot p_T)=p_Td^i$ and $\rho_{R\ot R}=(1\ot c\ot 1)(\rho\ot\rho )=(1\ot c\ot 1)(p_S\ot p_T\ot p_S\ot p_T)(\Delta_R\ot\Delta_R)=(p_S\ot p_S\ot p_T\ot p_T)\Delta_{R\ot R}$. Moreover,
\begin{eqnarray*}
\del^ifu & = & \del^i(fe_S*fe_T)u=\del^i(fe_Su*fe_Tu)=\del^ife_Su*\del^ife_Tu \\
 & = & \del^ifi_Su_Sp_S*\del^ifi_Tu_Tp_T=(\del^ifi_Su_S\ot\del^ifi_Tu_T)\rho_{R\ot R}
 \end{eqnarray*}
 as well.

Since $e_T*e_S=e_S*e_T=1_K$ on $K$ and hence $\del^ifu=\del^ife_Su*\del^ife_Tu=\del^ife_Tu*\del^ife_Su$, and since $\del fu (v\ot v)=\del^1fsu(v\ot v)$,  it follows that
$$\del fu=\del fe_Su*\del fe_Tu$$
as required.
\end{proof}

A comparison with Hochschild cohomology can be obtained inductively via a generalized \lq exponential' map, making use of the isomorphism
$$\begin{CD}
\Der_G(K,k)  @>\delta_{hoch}>> H^2_G(B,k) \\
@V\cong VV  @V\cong VV \\
\Der_G(K_S,k)\oplus\Der_G(K_T,k) @>\delta_{hoch}\oplus\delta_{hoch}>> H^2_G(B,k)
\end{CD}$$
and Proposition \ref{tp} to get a commutative square
$$\begin{CD}
\Der_G(K,k)  @>\delta_{hoch}>> H^2_G(B,k) \\
@V\exp VV  @V\Exp VV \\
\Alg_G(K,k) @>\delta >> \H^2_G(B,k)
\end{CD}$$
which says that
$$\delta e^{d_S+d_T}=\delta e^{d_S}*\delta e^{d_T}=\Exp_S(\delta_{hoch}d_S)*\Exp_T(\delta_{hoch}d_T)=\Exp (\delta_{hoch}(d_S+d_T))$$
by extending the notation naturally. In particular, if $f\in\Alg_G(K,k)$ and $\sigma =\delta f$ then
$$\sigma (x_i^m\ot x_j^n)=fsu (x_i^mx_j^n) =\left\{ \begin{array}{ll}
fsu(z_i)  & \mbox{, if $i=j$ and $m+n=N_i$} \\
n!_{q_i} fs(z_{ji})^n  &  \mbox{, if $i>j$ linkable and $m=n$} \\
0 & \mbox{, otherwise.}
\end{array}
\right. $$

\noindent
{\bf Remark.} Observe that Proposition \ref{tp} holds for any special diagram of finite Cartan type, provided that a  $K$-bimodule coalgebra retraction $u: R\to K$ exists. This is because $\del^i(fe_Su)*\del^j(fe_Tu)=\del^j(fe_Tu)*\del^i(fe_Su)$ for $f\in\Alg_G(K,k)$ and $i\le j$ if $S$ and $T$ are not linkable.

\subsection{The connected case}  Let $\mathcal D$ be a special connected datum of finite Cartan type with Cartan matrix $(a_{ij})$.  The vector space $V=V(\mathcal D)$ can also be viewed as a crossed module in $^{\mathbf Z[I]}_{\mathbf Z[I]}YD$, where $\mathbf Z[I]$ is the free abelian group on the set of simple roots $I=\{ \alpha_1, \ldots , \alpha_{\theta}\}$. The $\mathbf Z[I]$-degree of a word $x=x_{i_1}x_{i_2}\ldots x_{i_n}$ in the tensor algebra $mathcal a(V)$ is defined by $\deg (x)=\sum_{i=1}^{\theta}n_i\alpha_i$, where $n_i$ is the number of ocurrencies of $x_i$ in $x$. The Weyl group $W\subset \Aut (\mathbf Z[I])$ is generated by the automorphisms $s_i$ defined by $s_i(\alpha_j)=\alpha_j-a_{ij}\alpha_i$. The root system $\Phi =\cup_{i=1}^{\theta}W(\alpha_i )$ is the union of the orbits of simple roots in $\mathbf [I]$, and 
$$\Phi^+=\{ \alpha=\sum_{i=1}^{\theta}n_i\alpha_i\in\Phi | n_i\ge 0\}$$
is the set of positive roots. The Hopf algebra $\mathcal A(V)$, the quotient Hopf algebra $R(\mathcal D)=\mathcal
A(V)/(ad^{1-a_{ij}}x_i(x_j)|1\le i\ne j\le\theta )$ and its Hopf
subalgebra $K(\mathcal D)$ generated by $\setst{
x_{\alpha}^N}{\alpha\in\Phi^+}$, as well as the Nichols algebra $B(V)=R(V)/(x_{\alpha}^N)$, are all Hopf algebras in $^{\mathbf Z[I]}_{\mathbf Z[I]}YD$. In particular, their comultiplications are $\mathbf Z[I]$-graded. By construction, for
$\alpha\in\Phi^+$, the root vector $x_{\alpha}$ is $\mathbf Z[I]$-homogeneous of $\mathbf Z[I]$-degree $\alpha$,
so that $\delta (x_{\alpha})=g_{\alpha}\ot x_{\alpha}$ and $gx_{\alpha}=\chi_{\alpha}(g)x_{\alpha}$. For 
$1\le l\le p$ and for $a =(a_1, a_2,\ldots ,
a_p)\in\mathbf N^p$ write $\underline a=\sum_{i=1}^pa_i\beta_i$ and
$$g^a=g_1^{a_1}g_2^{a_2}\ldots g_p^{a_p}\in G\ ,\
\chi^a=\chi_1^{a_1}\chi_2^{a_2}\ldots\chi_p^{a_p}\in\tilde G\ , \
x^a=x_{\beta_1}^{a_1}x_{\beta_2}^{a_2}\ldots x_{\beta_p}^{a_p}\in R(\mathcal D).$$ In particular, for $e_l=(\delta_{kl})_{1\le k\le p}$, where
$\delta_{kl}$ is the Kronecker symbol, $\underline e_l=\beta_l$ and
$x^{e_l}=x_{\beta_l}$ and $x^{Ne_l}=x_{\beta_l}^N=z_l$ for $1\le l\le p$. In this notation
$$\{ x^a|0\le a_i\}\ ,\ \{ z^b|0\le b_i\}\ ,\ \{x^a|0\le a_i<N\}$$
form a PBW-basis for $R(V)$, $K(V)$ and $B(V)$, respectively. The height of $\alpha =\sum_{i=1}^{\theta}n_i\alpha_i\in\mathbf Z[I]$ is defined to be
the integer $ht(\alpha )=\sum_{i=1}^{\theta}n_i$. Observe that if
$a, b, c\in\mathbf N^p$ and $\underline a=\underline b+\underline c$ then
$$g^a=g^bg^c\ ,\ \chi^a = \chi^b\chi^c\ \rm{and}\ ht(\underline
b)<ht(\underline a)\ \rm{if}\ \underline c\ne 0.$$ By  \cite{AS} (Theorem 2.6), the sets
$$\{ z^b | 0\le b_i\}\ ,\  \{ x^az^b|0\le a_i<N, 0\le b_j\}\ ,\ \{x^a | 0\le a_i < N\}$$
form a basis for $K(V)$, $R(V)$ and $B(V)$, respectively.  The squares
$$\begin{CD}
K(V) @>\kappa >> R(V) @.  \quad , \quad @. K\# kG @>\kappa\# 1>>  R\# kG \\
@V\ep VV  @V\pi VV      @.                    @V\ep\# 1VV                             @V\pi\# 1 VV \\
k @>\iota >> B(V)             @.  \quad\ \quad @. kG @>\iota\# 1 >> B\# kG
\end{CD}$$
are pushout squares of braided Hopf algebras and their bosonizations, respectively. 
 
Moreover, the  $K$-module isomorphism $\vartheta :R\to B\ot K$ given by $\vartheta (x^az^b)=x^a\ot z^b$, can be used to get a $K$-module retraction $u=(\ep\ot 1)\vartheta :R(V)\to K(V)$, $u(x^az^b) = \ep (x^a)z^b)$, for the inclusion of $\kappa :K(V)\to R(V)$. Thus, the conditions for the $5$-term sequence in Hochschild cohomology are satisfied. The connecting map
$$\delta_{hoch} :\Der_G(K,k)\to H^2_G(B,k)$$
which is injective since $\Der_G(R,k)=0$, is such that $\delta_{hoch} d(\pi\ot\pi)=\del_{hoch}(du)=-dum_R$, where $\pi :R(V)\to B(V)$ is the canonical projection. The $K$-module map $u:R\to K$ just defined is not a coalgebra map in general.

Observe that $K=k[z_{\alpha} |\alpha\in\Phi^+]$ is a polynomial algebra, since by our assumption $\chi_i^N=\ep $ for $1\le i\le \theta $. The algebra isomorphism $\rho :\oplus_{\alpha\in\Phi^+}K_{\alpha}\to K$, given by $\rho (z_{\alpha_1}^{n_1}\ot z_{\alpha_2}^{n_2}\ot\ldots\ot z_{\alpha_p}^{n_p})=z_{\alpha_1}^{n_1}z_{\alpha_2}^{n_2}\ldots z_{\alpha_p}^{n_p}$, induces a commutative diagram
$$\begin{CD}
\Der_G(K,k) @>\rho_{\Der} >> \oplus_{\alpha\in\Phi^+}\Der_G(K_{\alpha},k) \\
@V\Exp VV  @V\exp VV \\
\Alg_G(K,k)  @>\rho_{\Alg} >> \times_{\alpha\in\Phi^+}\Alg_G(K_{\alpha},k)
\end{CD}$$
of sets, with $\rho_{\Der}(d)=(di_{\alpha})$, $\rho_{\Alg}(f)=(fi_{\alpha})$, $\exp ((d_{\alpha}))= (e^{d_{\alpha}})$ and $\Exp (d)=\rho_{\Alg}^{-1}\exp\rho_{\Der}$, where $i_{\alpha}:K_{\alpha}\to K$ and $p_{\alpha};K\to K_{\alpha}$ are the obvious canonical injections and projections. This means more explicitly that
$$\Exp (d)(z_{\alpha_1}^{n_1}z_{\alpha_2}^{n_2}\ldots z_{\alpha_p}^{n_p})=e^{di_1}(z_{\alpha_1}^{n_1})e^{di_2}(z_{\alpha_2}^{n_2})\ldots e^{di_p}(z_{\alpha_p}^{n_p})$$
for $d\in\Der_G(K,k)$.

If $\kappa :K\to R$ has a $K$-module coalgebra retraction $u_{\infty}:R\to K$ then Theorem \ref{5-term} is applicable, and the diagram
$$\begin{CD}
\Der_G(K,k) @>\delta_{hoch}>>  H_G^2(B,k) \\
@V\Exp VV  @. \\
\Alg_G(K,k)  @>\delta >>  \H_G^2(B,k)
\end{CD}$$
connects the relevant part of the Hochschild cohomology $H^2_G(B,k)$ to the multiplicative cohomology $\H^2_G(B,k)$.

\begin{Proposition} \label{injective} Let $V$ be a special (connected) diagram of finite Cartan type. If  $K(V)$ is a $K$-module coalgebra retract in $R(V)$ then the connecting map
$$\delta :\Alg_G(K,k)\to \H_G^2(B,k)$$
is injective. 
\end{Proposition}

\begin{proof} The simple root vectors $x_{\alpha}$, where $\alpha\in\Phi^+$ is a simple root, generate $R$ as an algebra. Moreover, $f(x_{\alpha})=f(gx_{\alpha})=\chi_{\alpha}(g)f(x_{\alpha})$ for every $g\in G$ and every $f\in\Alg_G(R,k)$. It follows that $\Alg_G(R,k)=\{\ep\}$, since $q_{\alpha}=\chi_{\alpha}(g_{\alpha})$ is a non-trivial root of unity for every simple root $\alpha\in\Phi^+$.

Now suppose that $\delta f=\delta f'$ in $\H_G^2(B,k)$ for some $f,f'\in\Alg_G(K,k)$. The representing cocycles $\sigma$ and $\sigma '$ are equivalent, so that $\sigma '=\del^0\chi *\del^2\chi *\sigma *\del^1\chi^{-1}$ for some $\chi\in\Reg_G(B,k)$. It follows that
\begin{eqnarray*}
\del^0f'u*\del^2f'u*\del^2f'su & = & \del f'u=(\pi\ot\pi )^*\sigma '=(\pi\ot\pi )^*(\del^0\chi *\del^2\chi *\sigma *\del^1\chi^{-1}) \\
& = & \del^0(\pi^*\chi )*\del^2(\pi^*\chi )*\del fu*\del^1(\pi^*\chi^{-1}) \\
& = & \del^0(\pi^*\chi )*\del^2(\pi^*\chi )*\del^0fu*\del^2fu*\del^1fsu*\del^1(\pi^*\chi^{-1}) \\
& = & \del^0(\pi^*\chi )*\del^0fu*\del^2(\pi^*\chi )*\del^2fu*\del^1fsu*\del^1(\pi^*\chi^{-1}) \\
& = & \del^0(\pi^*\chi *fu)*\del^2(\pi^*\chi *fu)*\del^1(fsu*\pi^*\chi^{-1})
\end{eqnarray*}
since the $\im\del^0$ and $\im\del^2$ commute elementwise. so that $\del^2(\pi^*\chi )*\del^0fu=(\pi^*\chi\ot\ep\ot\ep\ot fu)\Delta_{R\ot R}=(\ep\ot fu\ot\pi^*\chi\ot\ep )\Delta_{R\ot R}=\del^0fu*\del^2(\pi^*\chi )$. This means, again using the elementwise commutativity of $\im\del^0$ and $\im\del^2$, that
\begin{eqnarray*}
\del^1(fsu*\pi^*\chi^{-1}*f'u) & = & \del^1(fsu*\pi^*\chi^{-1})*\del^1f'u  \\
& = & \del^2(\pi^*\chi *fu)^{-1}*\del^0(\pi^8\chi *fu)^{-1}\del^of'u*\del^2f'u \\
& = & \del^0(fsu*\pi*\chi^{-1})\del^0f'u\del^2(fsu*\pi^*\chi^{-1})\del^2f'u \\
& = & \del^0(fsu*\pi^*\chi^{-1}*f'u)*\del^2(fsu*\pi^*\chi^{-1}*f'u)
\end{eqnarray*}
so that $fsu*\pi^*\chi^{-1}*f'u\in\Alg_G(R,k)=\{\ep\}$ and then $f'u=\pi^*\chi*fu$. But then
$$f'=f'u\kappa =(\pi^*\chi*fu)\kappa =fu\kappa*\chi\pi\kappa =\chi\ep*f=\ep *f=f$$
as required.
\end{proof}

The multiplicative cocycle $\sigma$ representing the cohomology class $\delta f$ is given by
$$(\pi\ot\pi )^*\sigma =\del fu_{\infty} =\del^0fu_{\infty}*\del^2fu_{\infty}*\del^1fsu_{\infty}=(fu_{\infty}\ot fu_{\infty})*fsu_{\infty}m_R$$
or, equivalently $\sigma =\del fu_{\infty}(v\ot v)$, where $v:B\to R$ is the obvious linear section of the canonical projection $\pi :R\to B$.

\begin{Conjecture} For every special connected diagram  of finite Cartan type $V$ the braided Hopf subalgebra $K$ is a $K$-module coalgebra retract in $R$.
\end{Conjecture}

Here is a recursive procedure to verify the conjecture. Let $B_i$ be the linear span in $B$ of all ordered words involving root vectors of height $\le i$ only. For $i>1$ let $B_i^j\subset B_i$ be the linear span of all ordered monomials in $B_i$ containing at most $j$ distinct root vectors of height $i$. Then $B_i^j$ is a subcoalgebra of $B$. The inclusion $v_i^j:B_i^j\to R$ is not a coalgebra map, but $B_i^j\ot K\subset R$ is a subcoalgebra under the coalgebra structure inherited from $R$ (not the tensor product coalgebra structure). This gives a finite filtration $B_i\subseteq B_{i+1}^j\subseteq B_{i+1}$ of $B$ and $\cup_{i\ge 0}B_i=B$. Observe that $B_i^0=B_{i-1}$ and $B_i^j=B_i$ for some $j$.
\begin{itemize}
\item For $B_1$ let $u_1=\ep\ot 1:B_1\ot K \to K$, which is a coalgebra map.
\item Suppose a coalgebra retraction $u_{i+1}^j=m_K(\varphi_{i+1}^j\ot 1):B_{i+1}^j\ot K\to K$ has been constructed. Extend $\varphi_{i+1}^j$ linearly to $B_{i+1}^{j+1}$ by sending to zero all PBW-monomials involving more than $j$ distinct root vectors of height $ i+1$. For such a PBW-monomial $x\in B_{i+1}^{j+1}\setminus B_{i+1}^j$ find a $z\in K$ such that
$\Delta_Kz -z\ot 1-1\ot z=(u_{i+1}^j\ot u_{i+1}^j)\Delta_Rv_{i+1}^{j+1}x$. Now define $\varphi_{i+1}^{j+1}:B_{i+1}^{j+1}\to K$ by $\varphi_{i+1}^{j+1}(x)=z$ and $\varphi_{i+1}^{j+1}|_{B_{i+1}^j}=\varphi_{i+1}^j$. Then $u_{i+1}^{j+1}=m_K(\varphi_{i+1}^{j+1}\ot 1):B_{i+1}^{j+1}\ot K\to K$ is a $K$-module coalgebra map.
\item Since $B$ is finite dimensional $B=B_i^j$ for some pair $(i,j)$. Then $u_{\infty}=u_i^j=m_K(\varphi_i^j\ot 1)\vartheta :R\to B\ot K\to K$ is a retraction for the inclusion $\kappa :K\to R$.
\end{itemize}

\subsection{Type $A_2$} Here we have a crossed $kG$-module $V=kx_1\ot kx_2$ with coaction
 $\delta (x_i)=g_i\ot x_i$ and action $gx_i=\chi_i(g)x_i$, where $\chi_i(g_i)=q$ and $\chi_j(g_i)\chi_i(g_j)=q_{ij}q_{ji}=q^{-1}$. If $e_{12}=x_1$, $e_{23}=x_2$ and $e_{13}=[e_{12},e_{23}]=[x_1,x_2]$
 then $\{ e_{12}^me_{13}^ne_{23}^l|0\le m,n,l <N\}$, $\{ e_{12}^me_{13}^ne_{23}^l|0\le m,n,l\}$ and $\{ z_{12}^mz_{13}^nz_{23}^l|0\le m,n,l\}$, where $z_{ij}=e_{ij}^N$, form a basis for $B(V)$, $R(V)$ and $K(V)$, respectively. In this notation, taken from \cite{AS1}, the comultiplications in the bosonisations are determined by
 $$\Delta (e_{ij})=\sum_{i\le p\le j}\lambda_{ipj}e_{ip}g_pg_j\ot e_{pj},$$
 where $e_{ii}=1$ and
 $$\lambda_{ipj}= \left \{ \begin{array}{ll}
 1 & \mbox{, if $i=p$ or $p=j$ } \\
 1-q^{-1} & \mbox{, if $i\ne p\ne j$}
 \end{array} \right . $$
 
\begin{Proposition} For diagrams of type $A_2$ the Hopf subalgebra $K\subset R$ is a $K$-bimodule coalgebra retract, with retraction $u_{\infty}=u_2:R\to K$.
\end{Proposition}

\begin{proof} It will be necessary to deform the $K$-bimodule retraction $u=(\ep\ot 1)\vartheta :R\to K$ somewhat to make it a coalgebra map in this case. Observe that $u_1=\ep\ot 1:B_1\ot K\to K$ is a $K$-module coalgebra map. The following arguments show that its extension $u=\ep\ot 1:B\ot K\to K$ is not a coalgebra map. In $R\# kG$  we get
\begin{eqnarray*}
\Delta (e_{12}^me_{13}^ne_{23}^l) & = & \sum_{1\le p_i\le 2, 1\le q_j\le 3, 2\le r_k\le 3}\lambda_{1p_12}\ldots\lambda_{1p_m2}\lambda_{1q_13}\ldots\lambda_{1q_n3}\lambda_{2r_13}\ldots\lambda_{2r_l3} \\
 & & e_{1p_1}g_{p_12}\ldots e_{1p_m}g_{p_m2}e_{1q_1}\ldots e_{1q_n}g_{q_n3}e_{2r_1}g_{r_13}\ldots e_{2r_l}g_{r_l3} \\
  & & \ot e_{p_12}\ldots e_{p_m2}e_{q_13}\ldots e_{q_n3}e_{r_13}\ldots e_{r_l3}
 \end{eqnarray*}
which contains the term $\lambda_{123}^n\chi_{12}^{{n\choose 2}}(g_{23})e_{12}^{m+n}g_{23}^{n+l}\ot e_{23}^{n+l}$, the only term that may make $(u\ot u)\Delta (e_{12}^me_{13}^ne_{23}^l)\ne 0$. In particular, if $m+n=N=n+l$ then $m=l$ and this term
 $$\lambda_{123}^n\chi_{12}^{{n\choose 2}}(g_{23})e_{12}^Ng_{23}^N\ot e_{23}^N$$
 is a non-zero element in $(K\# kG)\ot (K\# kG)$. It follows directly that
 \begin{eqnarray*}
 (u\ot u)\Delta (e_{12}^me_{13}^ne_{23}^l) & = u(e_{12}^me_{13}^ne_{23}^l)\ot 1+g_{12}^mg_{13}^ng_{23}^l\ot u(e_{12}^me_{13}^ne_{23}^l) \\
  & +\lambda_{123}^n\chi_{12}^{n\choose 2}(g_{23})u(e_{12}^{m+n})g_{23}^{n+l}\ot u(e_{23}^{n+l})
  \end{eqnarray*}
for $0\le m, n, l < N$. In particular, $(u\ot u)\Delta (e_{12}^me_{13}^ne_{23}^l)\ne 0$ if and only if $m+n=N=n+l$, and then  $$(u\ot u)\Delta (e_{12}^{N-n}e_{13}^ne_{23}^{N-n})=\lambda_{123}^n\chi_{12}(g_{23})^{n\choose 2}z_{12}h_{23}\ot z_{23}$$
while $u(e_{12}^me_{13}^ne_{23}^l)=0$. The $K$-bimodule retraction $u:R\to K$ defined by $u(x^az^b)=\ep (x^a)z^b$ is therefore not a coalgebra map. To remedy this situation, observe that
$$\Delta (z_{13})=z_{13}\ot 1+h_{13}\ot z_{13}+(1-q^{-1})^N\chi_{12}^{N\choose 2}(g_{23})z_{12}h_{23}\ot z_{23}$$
and define $u_2:R\to K$ by
$$u_2(e_{12}^me_{13}^ne_{23}^lz)=\delta^m_l\delta^{m+n}_N(1-q^{-1})^{n-N}\chi_{12}(g_{23})^{{n\choose 2}-{N\choose 2}}z_{13}z$$
for $z\in K$. Observe that $u_2=m_K(\varphi_2\ot 1)\vartheta :R\to B\ot K\to K\ot K\to K$, where $\varphi_2 :B\to K$ is given by
$$\varphi_2 (e_{12}^me_{13}^ne_{23}^l)=(1-q^{-1})^{-m}\chi_{12}(g_{23})^{{n\choose 2}-{{m+n}\choose 2}} u(e_{12}^{m-t}e_{13}^{n+t}e_{23}^{l-t}),$$
with $t=\min{(m,l)}$.
It then follows by construction that $u_{\infty}=u_2 :R\to K$ is a $K$-bimodule coalgebra retraction for $\kappa :K\to R$.
\end{proof}

The connecting map $\delta :\Alg_G(K,k)\to \H_G^2(B,k)$ guaranteed  by Theorem \ref{5-term} is injective by Proposition \ref{injective}, since $\Alg_G(R,k)=\{ \ep\}$ and since all elements of $\im\del^0$ commute with those of $\im\del^2$. The resulting cocycle deformations account for all liftings of $B\# kG$.

Results for type $A_n$, $n>2$, and type $B_2$ are in the pipeline. They will be a subject of a forthcoming paper.

\end{document}